\documentclass{mcom-l}

\usepackage{url}
\usepackage{multirow}
\usepackage{amsthm}
\usepackage{amsmath}
\usepackage{amssymb}
\newtheorem{Lem}{Lemma}

\newtheorem{theorem}{Theorem}

\numberwithin{equation}{section}
\begin{document}
\title[Error in the zero-counting formulae]{An improved upper bound for the error in the zero-counting formulae for Dirichlet $L$-functions and Dedekind zeta-functions}
\author{T. S. Trudgian}
\address{Mathematical Sciences Institute, The Australian National University, Canberra, Australia, 0200}
\email{timothy.trudgian@anu.edu.au}
\thanks{Supported by Australian Research Council DECRA Grant DE120100173.}


\date{9 October, 2013.}


\begin{abstract}
This paper contains new explicit upper bounds for the number of zeroes of Dirichlet $L$-functions and Dedekind zeta-functions in rectangles.
\end{abstract}

\maketitle

\section{Introduction and Results}
This paper pertains to the functions $N(T, \chi)$ and $N_{K}(T)$, respectively the number of zeroes $\rho = \beta + i\gamma$ of $L(s, \chi)$ and of $\zeta_{K}(s)$ in the region $0<\beta <1$ and $|\gamma|\leq T$. The purpose of this paper is to prove the following two theorems.

\begin{theorem}\label{McCurleyThm}
Let $T\geq 1$ and $\chi$ be a primitive nonprincipal character modulo $k$. Then
\begin{equation}\label{Thm11}
\bigg|N(T, \chi) - \frac{T}{\pi}\log\frac{kT}{2\pi e}\bigg| \leq 0.317\log kT + 6.401.
\end{equation}
In addition, if the right side of (\ref{Thm11}) is written as $C_{1}\log kT + C_{2}$, one may use the values of $C_{1}$ and $C_{2}$ contained in Table \ref{PrTable}.

\end{theorem}
\begin{theorem}\label{ZetaThm}
Let $T\geq 1$ and $K$ be a number field with degree $n_{K} = [K: \mathbb{Q}]$ and absolute discriminant $d_{K}$. Then
\begin{equation}\label{Thm12}
\bigg|N_{K}(T) - \frac{T}{\pi}\log\left\{ d_{K} \left( \frac{T}{2\pi e}\right)^{n_{K}}\right\}\bigg|\leq 0.317\left\{ \log d_{K} + n_{K}\log T\right\} + 6.333n_{K} + 3.482.
\end{equation}
In addition, if the right side of (\ref{Thm12}) is written as $D_{1}\left\{ \log d_{K} + n_{K}\log T\right\} + D_{2} n_{K} + D_{3}$, one may use the values of $D_{1}, D_{2}$ and $D_{3}$ contained in Table \ref{PrTable2}.
\end{theorem}

Theorem \ref{McCurleyThm} and Table \ref{PrTable} improve on a result due to McCurley \cite[Thm 2.1]{McCurley}; Theorem \ref{ZetaThm} and Table \ref{PrTable2} improve on a result due to Kadiri and Ng \cite[Thm 1]{NgKad2012}. The values of $C_{1}$ and $D_{1}$ given above are less than half of the corresponding values in \cite{McCurley} and \cite{NgKad2012}. The improvement is due to Backlund's trick --- explained in \S \ref{BT} --- and some minor optimisation.

Explicit expressions for $C_{1}$ and $C_{2}$ and for $D_{1}, D_{2}$ and $D_{3}$ are contained in (\ref{3.15}) and (\ref{C2f1}) and in (\ref{FinalD}) and (\ref{FinalD2}). These contain a parameter $\eta$ which, when varied, gives rise to Tables \ref{PrTable} and \ref{PrTable2}. The values in the right sides of (\ref{Thm11}) and (\ref{Thm12}) correspond to $\eta = \frac{1}{4}$ in the tables. Note that some minor improvement in the lower order terms is possible if $T\geq T_{0}>1$; Tables \ref{PrTable} and \ref{PrTable2} give this improvement when $T\geq 10$.
\begin{table}[h]
\caption{$C_{1}$ and $C_{2}$ in Theorem \ref{McCurleyThm} and in \cite{McCurley} for various values of $\eta$ }
\label{PrTable}
\centering
\begin{tabular}{c c c c c c }
\hline\hline

$\qquad\eta\qquad$  & \multicolumn{2}{c} \textrm{McCurley \cite{McCurley}\qquad} & \multicolumn{2}{c}{When $T\geq 1$\qquad\qquad}  & When $T\geq 10$\\
& $C_{1}$ &  $C_{2}$ & $C_{1}$ &  $C_{2}$ & $C_{2}$ \\[0.5 ex] \hline 
0.05 &0.506 &16.989 &0.248 & 9.339 &8.666     \\
0.10 & 0.552 & 13.202 & 0.265 & 8.015 & 7.311    \\
0.15 & 0.597 &11.067 & 0.282 &7.280   & 6.549   \\
0.20 & 0.643 &9.606 &0.300 & 6.778    &6.021  \\
0.25 & 0.689 & 8.509 & 0.317 & 6.401     & 5.616 \\
0.30 & 0.735 &7.641 &0.334  &6.101  &5.288  \\
0.35 & 0.781 & 6.929 & 0.351 & 5.852   & 5.011  \\
0.40 & 0.827 & 6.330 & 0.369 & 5.640  & 4.770  \\
0.45 & 0.873 & 5.817 & 0.386 & 5.456    & 4.556 \\
0.50 & 0.919 & 5.370 & 0.403 & 5.294   & 4.363  \\ \hline\hline
\end{tabular}
\end{table}

\begin{table}[h]
\caption{$D_{1}, D_{2}$ and $D_{3}$ in Theorem \ref{ZetaThm} and in \cite{NgKad2012} for various values of $\eta$ }
\label{PrTable2}
\centering
\begin{tabular}{c c c c c c c c c c}
\hline\hline

$\eta$  & \multicolumn{3}{c} \textrm{Kadiri and Ng \cite{NgKad2012}}  &  & \multicolumn{2}{c}{When $T\geq 1$}  & \multicolumn{2}{c} {When $T\geq 10$}\\
& $D_{1}$ & $D_{2}$ & $D_{3}$ & $D_{1}$ &  $D_{2}$ & $D_{3}$ & $D_{2}$ & $D_{3}$ \\[0.5 ex] \hline 
0.05 &0.506& 16.95& 7.663& 0.248& 9.270& 3.005& 8.637& 2.069 \\
0.10 & 0.552& 13.163& 7.663& 0.265& 7.947& 3.121& 7.288& 2.083   \\
0.15 & 0.597& 11.029& 7.663& 0.282& 7.211& 3.239& 6.526& 2.099   \\
0.20 &  0.643& 9.567& 7.663& 0.300& 6.710& 3.359& 5.997& 2.116  \\
0.25 & 0.689& 8.471& 7.663& 0.317& 6.333& 3.482& 5.593& 2.134\\
0.30 & 0.735& 7.603& 7.663& 0.334& 6.032& 3.607& 5.265& 2.153  \\
0.35 & 0.781& 6.891& 7.663& 0.351& 5.784& 3.733& 4.987& 2.173\\
0.40 &0.827& 6.292& 7.663& 0.369& 5.572& 3.860& 4.746& 2.193  \\
0.45 &0.873& 5.778& 7.663& 0.386& 5.388& 3.988& 4.532& 2.215  \\
0.50 & 0.919& 5.331& 7.663& 0.403& 5.225& 4.116& 4.339& 2.238 \\ \hline\hline
\end{tabular}
\end{table}

Explicit estimation of the error terms of the zero-counting function for $L(s, \chi)$ is done in \S \ref{Ds}. Backlund's trick is modified to suit Dirichlet $L$-functions in \S \ref{BT}. Theorem \ref{McCurleyThm} is proved in \S \ref{LemSec}. Theorem \ref{ZetaThm} is proved in \S \ref{zetasec}.

The Riemann zeta-function, $\zeta(s)$, is both a Dirichlet $L$-function (albeit to the principal character) and a Dedekind zeta-function. The error term in the zero counting function for $\zeta(s)$ has been improved, most recently, by the author \cite{TrudS}. One can estimate the error term in the case of $\zeta(s)$ more efficiently owing to explicit bounds on $\zeta(1+it)$, for $t\gg 1$. It would be of interest to see whether such bounds for $L(1+ it, \chi)$ and $\zeta_{K}(1+ it)$ could be produced relatively easily --- this would lead to an improvement of the results in this paper.

\section{Estimating $N(T, \chi)$}\label{Ds}
Let $\chi$ be a primitive nonprincipal character modulo $k$, and let $L(s, \chi)$ be the Dirichlet $L$-series attached to $\chi$. Let $a= (1-\chi(-1))/2$ so that $a$ is $0$ or $1$ according as $\chi$ is an even or an odd character. Then the function
\begin{equation}\label{xie}
\xi(s, \chi) = \left(\frac{k}{\pi}\right)^{(s+a)/2} \Gamma\left(\frac{s+a}{2}\right)L(s, \chi),
\end{equation}
is entire and satisfies the functional equation
\begin{equation}\label{xif}
\xi(1-s, \overline{\chi}) = \frac{i^{a}k^{1/2}}{\tau(\chi)}\xi(s, \chi),
\end{equation}
where $\tau(\chi) = \sum_{n=1}^{k}\chi(n)\exp(2\pi in/k)$.

Let $N(T, \chi)$ denote the number of zeroes $\rho = \beta + i\gamma$ of $L(s, \chi)$ for which $0< \beta < 1$ and $|\gamma|\leq T$. For any $\sigma_{1}>1$ form the rectangle $R$ having vertices at $\sigma_{1} \pm iT$ and $1-\sigma_{1} \pm iT$, and let $\mathcal{C}$ denote the portion of the boundary of the rectangle in the region $\sigma \geq \frac{1}{2}$.  From Cauchy's theorem and (\ref{xif})  one deduces that
\begin{equation*}\label{N1}
N(T, \chi) = \frac{1}{\pi} \Delta_{\mathcal{C}} \arg \xi(s, \chi).
\end{equation*}
Thus
\begin{equation}\label{N2}
\begin{split}
N(T, \chi) &= \frac{1}{\pi}\left\{ \Delta_{\mathcal{C}}\arg\left(\frac{k}{\pi}\right)^{(s+a)/2} + \Delta_{\mathcal{C}}\arg\Gamma\left(\frac{s+a}{2}\right) + \Delta_{\mathcal{C}}\arg L(s, \chi)\right\}\\
&= \frac{T}{\pi} \log\frac{k}{\pi} + \frac{2}{\pi} \Im\log\Gamma\left(\frac{1}{4} + \frac{a}{2} + i\frac{T}{2}\right) + \frac{1}{\pi}\Delta_{\mathcal{C}}\arg L(s, \chi).
\end{split}
\end{equation}
To evaluate the second term on the right-side of (\ref{N2}) one needs an explicit version of Stirling's formula. Such a version is provided in \cite[p.\ 294]{Olver}, to wit
\begin{equation}\label{Stirling}
\log\Gamma(z) = (z-\frac{1}{2})\log z - z + \frac{1}{2} \log 2\pi + \frac{\theta}{6|z|},
\end{equation}
which is valid for $|\arg z| \leq \frac{\pi}{2}$, and in which $\theta$ denotes a complex number satisfying $|\theta| \leq 1$.
Using (\ref{Stirling}) one obtains
\begin{equation}\label{Storl}
\begin{split}
\Im\log\Gamma\left(\frac{1}{4} + \frac{a}{2} + i\frac{T}{2}\right) = & \frac{T}{2} \log\frac{T}{2e} + \frac{T}{4} \log\left( 1+ \frac{(2a+1)^{2}}{4T^{2}}\right) \\
&+ \frac{2a-1}{4}\tan^{-1}\left(\frac{2T}{2a +1}\right) + \frac{\theta}{3|\frac{1}{2} + a+ iT|}.
\end{split}
\end{equation}
Denote the last three terms in (\ref{Storl}) by $g(a, T)$. Using elementary calculus one can show that $|g(0,T)|\leq g(1, T)$ and that $g(1, T)$ is decreasing for $T\geq 1$.
This, together with (\ref{N2}) and (\ref{Storl}), shows that
\begin{equation}\label{3.11}
\bigg|N(T, \chi) - \frac{T}{\pi}\log\frac{kT}{2\pi e}\bigg| \leq \frac{1}{\pi}\big| \Delta_{\mathcal{C}}\arg L(s, \chi)\big| + \frac{2}{\pi}g(1, T).
\end{equation}

All that remains is to estimate $\Delta_{\mathcal{C}}\arg L(s, \chi)$. Write $\mathcal{C}$ as the union of three straight lines, viz.\ let $\mathcal{C} = \mathcal{C}_{1} \cup \mathcal{C}_{2} \cup \mathcal{C}_{3}$, where $\mathcal{C}_{1}$ connects $\frac{1}{2} - iT$ to $\sigma_{1} - iT$; $\mathcal{C}_{2}$ connects $\sigma_{1} - iT$ to $\sigma_{1} + iT$; and $\mathcal{C}_{3}$ connects $\sigma_{1} + iT$ to $\frac{1}{2} + iT$.  Since $L(\overline{s}, \chi) = \overline{L(s, \overline{\chi})}$ a bound for the integral on $\mathcal{C}_{3}$ will serve as a bound for that on $\mathcal{C}_{1}$. Estimating the contribution along $\mathcal{C}_{2}$ poses no difficulty since
\begin{equation*}
|\arg L(\sigma_{1} + it, \chi)| \leq |\log L(\sigma_{1} + it, \chi)| \leq \log\zeta(\sigma_{1}).
\end{equation*}
To estimate $\Delta_{\mathcal{C}_{3}}\arg L(s, \chi)$ define
\begin{equation}\label{wwbw}
f(s) = \frac{1}{2} \{L(s+ iT, \chi)^{N} + L(s-iT, \overline{\chi})^{N}\},
\end{equation}
for some positive integer $N$, to be determined later. Thus $f(\sigma) = \Re L(\sigma + iT, \chi)^{N}$. Suppose that there are $n$ zeroes of $\Re L(\sigma + iT, \chi)^{N}$ for $\sigma\in [\frac{1}{2}, \sigma_{1}]$. These zeroes partition the segment into $n+1$ intervals. On each interval $\arg L(\sigma + iT, \chi)^{N}$ can increase by at most $\pi$. Thus
\begin{equation*}
|\Delta_{\mathcal{C}_{3}} \arg L(s, \chi)| = \frac{1}{N} |\Delta_{\mathcal{C}_{3}} \arg L(s, \chi)^{N}| \leq \frac{(n+1)\pi}{N},
\end{equation*}
whence (\ref{3.11}) may be written as
\begin{equation}\label{stp}
\bigg|N(T, \chi) - \frac{T}{\pi}\log\frac{kT}{2\pi e}\bigg| \leq \frac{2}{\pi}\left\{\log\zeta(\sigma_{1}) + g(1,T)\right\} + \frac{2(n+1)}{N}.
\end{equation}
One may estimate $n$ with Jensen's Formula.
\begin{Lem}[Jensen's Formula]\label{JL}
Let $f(z)$ be holomorphic for $|z-a|\leq R$ and non-vanishing at $z=a$. Let the zeroes of $f(z)$ inside the circle be $z_{k}$, where $1\leq k\leq n$, and let $|z_{k} - a| = r_{k}$. Then
\begin{equation}\label{JF1}
\log\frac{R^{n}}{|r_{1} r_{2} \cdots r_{n}|} = \frac{1}{2\pi}\int_{0}^{2\pi} \log f(a+ Re^{i\phi})\, d\phi - \log |f(a)|.
\end{equation}
\end{Lem}
This is done in \S \ref{LemSec}.

\section{Backlund's Trick}\label{BT}
For a complex-valued function $F(s)$, and for  $\delta>0$ define $\Delta_{+}\arg F(s)$ to be the change in argument of $F(s)$ as $\sigma$ varies from $\frac{1}{2}$ to $\frac{1}{2} + \delta$, and define $\Delta_{-}\arg F(s)$ to be the change in argument of $F(s)$ as $\sigma$ varies from $\frac{1}{2}$ to $\frac{1}{2} - \delta$.

Backlund's trick is to show that if there are zeroes of $\Re F(\sigma + iT)^{N}$ on the line $\sigma \in[\frac{1}{2}, \sigma_{1}]$, then there are zeroes on the line $\sigma \in[1-\sigma_{1}, \frac{1}{2}]$. This device was introduced by Backlund in \cite{Backlund1918} for the Riemann zeta-function.

Following Backlund's approach one can prove the following general lemma.

\begin{Lem}\label{BL}
Let $N$ be a positive integer and let $T\geq T_{0}\geq 1$. Suppose that there is an upper bound $E$ that satisfies 
\begin{equation*}
| \Delta_{+}\arg F(s) + \Delta_{-}\arg F(s)| \leq E,
\end{equation*}
where $E = E(\delta, T_{0})$. Suppose further that there exists an $n\geq 3 + \lfloor NE/\pi\rfloor$ for which
\begin{equation}\label{C3gt}
n\pi \leq|\Delta_{\mathcal{C}_{3}} \arg F(s)^{N}| < (n+1)\pi.
\end{equation}
Then there are at least $n$ distinct zeroes of $\Re F(\sigma + iT)^{N}$, denoted by  $\rho_{\nu}= a_{\nu} + iT$ (where $1\leq \nu \leq n$ and $\frac{1}{2} \leq a_{n}< a_{n-1} < \cdots \leq \sigma_{1}$), such that the bound $|\Delta\arg F(s)^{N}| \geq \nu\pi$ is achieved for the first time when $\sigma$ passes over $a_{\nu}$ from above.

In addition there are at least $n-2 - \lfloor N E/\pi\rfloor$ distinct zeroes $\rho_{\nu}' = a_{\nu}' + iT$ (where $1\leq \nu \leq n-2$ and $1-\sigma_{1} \leq a_{1}' < a_{2}' < \cdots \leq \frac{1}{2}$).

Moreover
\begin{equation}\label{jenpre}
a_{\nu} \geq 1-a_{\nu}', \quad \textrm{for } \nu = 1, 2, \ldots, n-2 - \lfloor NE/\pi\rfloor ,
\end{equation}
and, if $\eta$ is defined by $\sigma_{1} = \frac{1}{2} + \sqrt{2}(\eta + \frac{1}{2})$, then
\begin{equation}\label{finalae}
\prod_{\nu = 1}^{n} |1 + \eta - a_{\nu}| \prod_{\nu = 1}^{n -2 - \lfloor NE/\pi\rfloor } |1 + \eta - a_{\nu}'|\leq (\tfrac{1}{2} + \eta)^{2n-2 - \lfloor NE/\pi\rfloor }.
\end{equation}
\end{Lem}
\begin{proof}
It follows from (\ref{C3gt}) that $|\arg F(s)^{N}|$ must increase as $\sigma$ varies from $\sigma_{1}$ to $\frac{1}{2}$. This increase may only occur if $\sigma$ has passed over a zero of $\Re F(s)^{N}$, irrespective of its multiplicity. In particular as $\sigma$ moves along $\mathcal{C}_{3}$
\begin{equation*}
|\Delta\arg F(s)^{N}| \geq \pi, 2\pi,\ldots,  n\pi.
\end{equation*}
Let $\rho_{\nu}= a_{\nu} + it$ denote the distinct zeroes of $\Re F(s)^{N}$ the passing over of which produces, for the first time, the bound $|\Delta\arg F(s)^{N}| \geq \nu\pi$. It follows that there must be $n$ such points, and that $\frac{1}{2} \leq a_{n} < a_{n-1}< \ldots < a_{2} < a_{1} \leq \sigma_{1}$. Also if $\frac{1}{2} + \delta \geq a_{\nu}$ then
\begin{equation}\label{gmcd}
|\Delta_{+}\arg F(s)^{N}| \geq  (n-\nu)\pi.
\end{equation}
For (\ref{gmcd}) is true when $\nu = n$ and so, by the definition of $\rho_{\nu}$, it is true for all $1\leq \nu\leq n$.

By the hypothesis in Lemma \ref{BL},
\begin{equation}\label{sas3e}
| \Delta_{+}\arg F(s)^{N} + \Delta_{-}\arg F(s)^{N}| \leq NE.
\end{equation}
When $\frac{1}{2} + \delta \geq a_{\nu},$ (\ref{gmcd}) and (\ref{sas3e}) show that
\begin{equation}\label{jargo}
|\Delta_{-}\arg F(s)^{N}|  \geq (n-\nu - NE/\pi)\pi,
\end{equation}
for $1 \leq \nu \leq n-2-\lfloor NE/\pi\rfloor $. When $\frac{1}{2} + \delta = a_{\nu}$ and $\nu= n-2-\lfloor NE/\pi\rfloor$, it follows from (\ref{jargo}) that $|\Delta_{-}\arg F(s)^{N}|  \geq \pi$. The increase in the argument is only possible if there is a zero of $\Re F(s)^{N}$ the real part of which is greater than $\frac{1}{2} - \delta = 1- a_{n-2-\lfloor NE/\pi\rfloor}$. Label this zero $\rho_{n-2-\lfloor NE/\pi\rfloor}' = a_{n-2-\lfloor NE/\pi\rfloor}' + iT$. Repeat the procedure when $\nu = n-3 -\lfloor NE/\pi\rfloor, \ldots, 2, 1$, whence (\ref{jenpre}) follows. This produces a positive number of zeroes in $[1-\sigma_{1}, \frac{1}{2}]$ provided that $n\geq 3 + \lfloor NE/\pi\rfloor$.


For zeroes $\rho_{\nu}$ lying to the left of $1+ \eta$ one has
\begin{equation*}
|1+ \eta - a_{\nu}||1+ \eta - a_{\nu}'| \leq  (1+ \eta - a_{\nu})(\eta + a_{\nu}),
\end{equation*}
by (\ref{jenpre}). This is a decreasing function for $a_{\nu}\in[\frac{1}{2}, 1+ \eta]$ and so, for these zeroes
\begin{equation}\label{lbt}
|1+ \eta - a_{\nu}||1+ \eta - a_{\nu}'|  \leq (\tfrac{1}{2} + \eta)^{2}.
\end{equation}
For zeroes lying to the right of $1+ \eta$ one has
\begin{equation*}
|1+ \eta - a_{\nu}||1+ \eta - a_{\nu}'|  \leq  (a_{\nu} - 1 - \eta)(\eta + a_{\nu}).
\end{equation*}
This is increasing with $a_{\nu}$ and so, for these zeroes
\begin{equation}\label{rbt}
|1+ \eta - a_{\nu}||1+ \eta - a_{\nu}'|  \leq \sigma_{1}^{2} - \sigma_{1} - \eta(1+ \eta).
\end{equation}
The bounds in (\ref{lbt}) and (\ref{rbt}) are equal\footnote{McCurley does not use Backlund's trick. Accordingly, his upper bounds in place of (\ref{lbt}) and (\ref{rbt}) are $\frac{1}{2} + \eta$ and $\sigma_{1} - 1 - \eta$. These are equal at $\sigma_{1} = \frac{3}{2} + 2\eta$, which is his choice of $\sigma_{1}$.} when $\sigma_{1} = \frac{1}{2} + \sqrt{2}(\eta + \frac{1}{2})$. Thus (\ref{finalae}) holds for $\sigma_{1} =  \frac{1}{2} + \sqrt{2}(\eta + \frac{1}{2})$. For the unpaired zeroes 
one may use the bound $|1+\eta - a_{\nu}| \leq \frac{1}{2} + \eta$, whence (\ref{finalae}) follows.
\end{proof}

\subsection{Applying Backlund's Trick}
Apply Jensen's formula on the function $F(s)$, with $a = 1+\eta$ and $R = r(\frac{1}{2}+ \eta)$, where $r>1$. Assume that the hypotheses of Lemma \ref{BL} hold. 
If $1+ \eta - r(\frac{1}{2} + \eta) \leq 1-\sigma_{1}$ then all of the $2n-1-\lfloor NE/\pi\rfloor$ zeroes of $\Re F(\sigma+ iT)^{N}$ are included in the contour. Thus the left side of (\ref{JF1}) is
\begin{equation}\label{MayDay}
\begin{split}
&\log\frac{\{r(\frac{1}{2}+ \eta)\}^{2n-2-\lfloor NE/\pi\rfloor}}{|1+\eta - a_{1}| \cdots |1+ \eta - a_{n}||1+ \eta - a'_{1}|\cdots |1+ \eta - a_{n-2-\lfloor NE/\pi\rfloor}'|} \\
&\geq (2n -2-\lfloor NE/\pi\rfloor)\log r,
\end{split}
\end{equation}
by (\ref{finalae}).
If the contour does not enclose all of the $2n-2-[NE/\pi]$ zeroes of $\Re F(\sigma + iT)^{N}$, then the following argument, thoughtfully provided by Professor D.R.\ Heath-Brown, allows one still to make a saving.

To a zero at $x+ it$, with $\frac{1}{2} \leq x \leq 1+ \eta$ one may associate a zero at $x' + it$ where, by (\ref{jenpre}), $1-x \leq x' \leq \frac{1}{2}$. Thus, for an intermediate radius, zeroes to the right of $\frac{1}{2}$ yet still close to $\frac{1}{2}$ will have their pairs included in the contour. Let $X$ satisfy $1 + \eta - (\frac{1}{2} + \eta)/r < X < \min\{1+ \eta, r(\frac{1}{2} + \eta) - \eta\}$. Since $r>1$, this guarantees that $X>\frac{1}{2}$. For a zero at $x+ it$ consider two cases: $x \geq X$ and $x <X$. 

In the former, there is no guarantee that the paired zero $x'+ it$ is included in the contour. Thus the zero at $x+ it$ is counted in Jensen's formula with weight
\begin{equation}\label{1stHBc}
\log \frac{r(\frac{1}{2} + \eta)}{1+ \eta - x} \geq \log \frac{r(\frac{1}{2} + \eta)}{1+ \eta - X}.
\end{equation}

Now, when $x<X$, the paired zero at $x'$ is included in the contour, since $1+ \eta - r(\frac{1}{2} + \eta) <1-X < 1-x \leq x'$. Thus, in Jensen's formula, the contribution is
\begin{equation}\label{uofx}
\begin{split}
\log \frac{r(\frac{1}{2} + \eta)}{1+ \eta - x} + \log \frac{r(\frac{1}{2} + \eta)}{1+ \eta - x'} &\geq \log \frac{r(\frac{1}{2} + \eta)}{1+ \eta - x} + \log \frac{r(\frac{1}{2} + \eta)}{\eta + x}\\
&= \log \frac{r^{2} (\frac{1}{2} + \eta)^{2}}{(1+ \eta - x)(\eta + x)}.
\end{split}
\end{equation}
The function appearing in the denominator of (\ref{uofx}) is decreasing for $x\geq \frac{1}{2}.$ Thus the zeroes at $x+ it$ and $x' + it$ contribute at least $2\log r$. 

Suppose now that there are $n$ zeroes in $[\frac{1}{2}, \sigma_{1}]$, and that there are $k$ zeroes the real parts of which are at least $X$. The contribution of all the zeroes ensnared by the integral in Jensen's formula is at least
\begin{equation*}
k\log \frac{r(\frac{1}{2} + \eta)}{1+ \eta - X} + 2(n-k)\log r = k\log\frac{(\frac{1}{2} + \eta)}{r(1+ \eta - X)} + 2n\log r \geq 2n\log r,
\end{equation*}
which implies (\ref{MayDay})

\subsection{Calculation of $E$ in Lemma \ref{BL}}
From (\ref{xie}) and (\ref{xif}) it follows that
\begin{equation*}\label{args}
\Delta_{+}\arg \xi(s, \chi) = - \Delta_{-}\arg \xi(s, \chi).
\end{equation*}
Since $\arg (\pi/k)^{-\frac{s+a}{2}} = -\frac{t}{2}\log (\pi/k)$ then $\Delta_{\pm} (\pi/k)^{-\frac{s+a}{2}} =0$, whence
\begin{equation*}\label{zetaarg}
| \Delta_{+}\arg L(s, \chi) + \Delta_{-}\arg L(s, \chi)| = | \Delta_{+}\arg\Gamma(\tfrac{s+a}{2}) + \Delta_{-}\arg\Gamma(\tfrac{s+a}{2})|.
\end{equation*}
Using 
(\ref{Stirling}) one may write
\begin{equation}\label{llde}
\bigg| \Delta_{+}\arg \Gamma  \left(\frac{s+a}{2}\right) + \Delta_{-}\arg\Gamma\left(\frac{s+a}{2}\right)\bigg| \leq  G(a, \delta, t),
\end{equation}
where
\begin{equation}\label{Gdef}
\begin{split}
G(a, \delta, t) = &\frac{1}{2}(a-\frac{1}{2} + \delta) \tan^{-1}\frac{a+ \frac{1}{2} + \delta}{t} + \frac{1}{2}(a-\frac{1}{2} - \delta) \tan^{-1}\frac{a+ \frac{1}{2} - \delta}{t}\\ & - (a-\frac{1}{2}) \tan^{-1} \frac{a+ \frac{1}{2}}{t}
 -\frac{t}{4}\log\left[1+ \frac{2\delta^{2}\{t^{2} - (\frac{1}{2} + a)^{2}\} + \delta^{4}}{\left\{t^{2} + (\frac{1}{2} + a)^{2}\right\}^{2}}\right]\\ 
&+\frac{1}{3}\left\{\frac{1}{|\frac{1}{2} + \delta +a + it|} + \frac{1}{|\frac{1}{2} - \delta +a + it|} + \frac{2}{|\frac{1}{2} +a + it|}\right\}.
\end{split}
\end{equation}
One can show that $G(a, \delta, t)$ is decreasing in $t$ and increasing in $\delta$, and that $G(1, \delta, t) \leq G(0, \delta, t)$. Therefore, since, in Lemma \ref{BL}, one takes $\sigma_{1} = \frac{1}{2} + \sqrt{2}(\frac{1}{2} + \eta)$ it follows that $\delta = \sqrt{2}(\frac{1}{2} + \eta)$, whence one may take
\begin{equation}\label{Etake}
E = G(0, \sqrt{2}(\tfrac{1}{2} + \eta), t_{0}),
\end{equation}
for $t\geq t_{0}$.

\section{Proof of Theorem 1}\label{LemSec}
First, suppose that $|\Delta_{\mathcal{C}_{3}} \arg L(s, \chi)^{N}| <  3 + \lfloor NE/\pi\rfloor$. Thus (\ref{3.11}) becomes
\begin{equation}\label{polk}
\bigg|N(T, \chi) - \frac{T}{\pi}\log\frac{kT}{2\pi e}\bigg| \leq \frac{2}{\pi}\left\{\log\zeta(\sigma_{1}) + g(1,T) + E\right\} + \frac{6}{N}.
\end{equation} 
Now suppose that $|\Delta_{\mathcal{C}_{3}} \arg L(s, \chi)^{N}| \geq  3 + \lfloor NE/\pi\rfloor$, whence Lemma \ref{BL} may be applied.

To apply Jensen's formula to the function $f(s)$, defined in (\ref{wwbw}), it is necessary to show that $f(1+\eta)$ is non-zero: this is easy to do upon invoking an observation due to Rosser \cite{Rossers}. Write $L(1+\eta + iT, \chi) = Ke^{i\psi}$, where $K> 0$. Choose a sequence of $N$'s tending to infinity for which $N\psi$ tends to zero modulo $2\pi$. Thus
\begin{equation}\label{Rost}
\frac{f(1+\eta)}{|L(1+\eta+iT, \chi)|^{N}} \rightarrow 1.
\end{equation}
Since $\chi$ is a primitive nonprincipal character then $f(s)$ is holomorphic on the circle. It follows from (\ref{JF1}) and (\ref{MayDay}) that
\begin{equation}\label{jenco1}
n \leq \frac{1}{4\pi\log r} J - \frac{1}{2\log r} \log |f(1+\eta)| +1 + \frac{NE}{2\pi},
\end{equation}
where
\begin{equation*}\label{71}
J = \int_{-\frac{\pi}{2}}^{\frac{3\pi}{2}} \log|f(1+\eta + r(\frac{1}{2} + \eta)e^{i\phi})|\, d\phi.
\end{equation*}
Write $J = J_{1} + J_{2}$ where the respective ranges of integration of $J_{1}$ and $J_{2}$ are $\phi\in[-\pi/2, \pi/2]$ and $\phi\in[\pi/2, 3\pi/2]$. For $\sigma>1$
\begin{equation}\label{ie1}
\frac{\zeta(2\sigma)}{\zeta(\sigma)} \leq |L(s, \chi)| \leq \zeta(\sigma),
\end{equation}
which shows that
\begin{equation}\label{J1f}
J_{1} \leq N\int_{-\pi/2}^{\pi/2}\log\zeta(1+\eta + r(\tfrac{1}{2}+ \eta)\cos\phi)\, d\phi.
\end{equation}
On $J_{2}$ use
\begin{equation*}
\log |f(s)| \leq N\log |L(s + iT, \chi)|,
\end{equation*}
and the convexity bound \cite[Thm 3]{Rademacher}
\begin{equation}\label{rconv}
|L(s, \chi)| \leq \left(\frac{k|s+1|}{2\pi}\right)^{(1+\eta - \sigma)/2} \zeta(1+\eta),
\end{equation}
valid for $-\eta \leq \sigma \leq 1+\eta$, where $0<\eta \leq \frac{1}{2}$, to show that
\begin{equation}\label{J2}
J_{2} \leq \pi N \log \zeta(1+\eta) + N\frac{r(\frac{1}{2} + \eta)}{2} \int_{\pi/2}^{3\pi/2}(-\cos\phi) \log\left\{\frac{kTw(T, \phi, \eta, r)}{2\pi}\right\}\, d\phi,
\end{equation}
where
\begin{equation}\label{wder}
\begin{split}
w(& T, \phi, \eta,r)^{2} = \\
&1+ \frac{2r(\frac{1}{2} + \eta)\sin\phi}{T} + \frac{r^{2}(\frac{1}{2} + \eta)^{2} + (2+ \eta)^{2} + 2r(\frac{1}{2} + \eta)(2+\eta)\cos\phi}{T^{2}}.
\end{split}
\end{equation}
For $\phi\in[\pi/2, \pi]$, the function $w(T, \phi, \eta, r)$ is decreasing in $T$; for $\phi\in[\pi, 3\pi/2]$ it is bounded above by $w^{*}(T, \phi, \eta,r)$ where
\begin{equation}\label{wder2}
w^{*}(T, \phi, \eta,r)^{2} = 1 + \frac{r^{2}(\frac{1}{2} + \eta)^{2} + (2+ \eta)^{2} + 2r(\frac{1}{2} + \eta)(2+\eta)\cos\phi}{T^{2}},
\end{equation}
which is decreasing in $T$.

To bound $n$ using (\ref{jenco1}) it remains to bound $-\log|f(1+\eta)|$. This is done by using (\ref{Rost}) and (\ref{ie1}) to show that
\begin{equation*}
-\log |f(1+\eta)| \rightarrow -N \log |L(1+\eta + iT)| \leq -N \log[\zeta(2+2\eta)/\zeta(1+\eta)].
\end{equation*}
This, together with (\ref{stp}), (\ref{polk}), (\ref{jenco1}), (\ref{J1f}), (\ref{J2}) and sending $N\rightarrow \infty$, shows that, when $T\geq T_{0}$
\begin{equation}\label{3.14}
\bigg|N(T, \chi) - \frac{T}{\pi}\log\frac{kT}{2\pi e}\bigg| \leq \frac{r(\frac{1}{2} + \eta)}{2\pi\log r} \log kT + C_{2},
\end{equation}
where
\begin{equation*}
\begin{split}
C_{2} = \frac{2}{\pi}& \left\{\log\zeta(\tfrac{1}{2} + \sqrt{2}(\tfrac{1}{2} + \eta)) + g(1, T) + \frac{E}{2}\right\} + \frac{3}{2\log r}  \log\zeta(1+\eta)\\ &- \frac{\log \zeta(2+2\eta)}{\log r} + \frac{1}{2\pi\log r}\int_{-\pi/2}^{\pi/2} \log\zeta(1+ \eta + r(\tfrac{1}{2}+ \eta)\cos\phi)\, d\phi\\
& +\frac{r(\frac{1}{2} + \eta)}{4\pi\log r}\bigg\{ -2\log 2\pi + \int_{\pi/2}^{\pi} (-\cos\phi)\log w(T_{0}, \phi, \eta, r)\, d\phi \\
&\qquad\qquad\qquad\qquad\qquad\qquad+ \int_{\pi}^{3\pi/2} (-\cos\phi)\log w^{*}(T_{0}, \phi, \eta, r)\, d\phi\bigg\}. 
\end{split}
\end{equation*}

\subsection{A small improvement}\label{Asi}
Consider that what is really sought is a number $p$ satisfying $-\eta \leq p <0$ for which one can bound $L(p+it, \chi)$, provided that $1+ \eta - r(\frac{1}{2} + \eta)\geq p$. Indeed the restriction that $p\geq-\eta$ can be relaxed by adapting the convexity bound, but, as will be shown soon, this is unnecessary. 

The convexity bound (\ref{rconv}) becomes the rather ungainly
\begin{equation*}
|L(s, \chi)| \leq \left\{\left(\frac{k|1+ s|}{2\pi}\right)^{(1/2 - p)(1 + \eta - \sigma)} \zeta(1 - p)^{1 + \eta - \sigma} \zeta(1+ \eta)^{\sigma - p}\right\}^{1/(1 + \eta - p)},
\end{equation*}
valid for $-\eta \leq p \leq \sigma \leq 1 + \eta$. Such an alternation only changes $J_{2}$, whence the coefficient of $\log kT$ in (\ref{3.14}) becomes 
\begin{equation*}
\frac{r(\frac{1}{2} + \eta)(\frac{1}{2} - p)}{\pi(1+ \eta -p)\log r}.
\end{equation*}
This is minimised when $r = (1 + \eta - p)/(1/2 + \eta)$, whence (\ref{3.14}) becomes
\begin{equation}\label{3.15}
\bigg|N(T, \chi) - \frac{T}{\pi}\log\frac{kT}{2\pi e}\bigg| \leq \frac{\frac{1}{2} - p}{\pi\log \left(\frac{1+ \eta - p}{1/2 + \eta}\right)} \log kT + C_{2},
\end{equation}
where
\begin{equation}\label{C2f1}
\begin{split}
&C_{2} = \frac{2}{\pi} \left\{\log\zeta(\tfrac{1}{2} + \sqrt{2}(\tfrac{1}{2} + \eta)) + g(1, T) + \frac{G(0, \sqrt{2}(\frac{1}{2} + \eta), T_{0})}{2}\right\} \\
&+ \frac{1}{\log\left(\frac{1+ \eta - p}{1/2 + \eta}\right)}\bigg\{  \frac{3}{2}\log\zeta(1+\eta) 
- \log \zeta(2+2\eta) +\frac{1}{\pi}\log \frac{\zeta(1-p)}{\zeta(1+ \eta)}\\
& + \frac{1}{2\pi}\int_{-\pi/2}^{\pi/2} \log\zeta(1+ \eta + (1+ \eta - p)\cos\phi)\, d\phi 
+\frac{\frac{1}{2} - p}{2\pi}\bigg(-2\log 2\pi \\
&+ \int_{\pi/2}^{\pi} (-\cos\phi)\log w(T_{0}, \phi, \eta, r)\, d\phi + \int_{\pi}^{3\pi/2} (-\cos\phi)\log w^{*}(T_{0}, \phi, \eta, r)\, d\phi\bigg)\bigg\},
\end{split}
\end{equation}
in which $g(1, T)$, $G(a, \delta, T_{0})$, $w$ and $w^{*}$ are defined in (\ref{Storl}), (\ref{Gdef}), (\ref{wder}) and (\ref{wder2}).

The coefficient of $\log kT$ in (\ref{3.15}) is minimal when $p=0$ and $r = \frac{1 + \eta}{1/2 + \eta}$. One cannot choose $p=0$ nor should one choose $p$ to be too small a negative number lest the term $\log \zeta(1-p)/\zeta(1+ \eta)$ become too large. Choosing $p=-\eta/7$ ensures that $C_{2}$ in (\ref{3.15}) is always smaller than the corresponding term in McCurley's proof. Theorem \ref{McCurleyThm} follows upon taking $T_{0} = 1$ and $T_{0} = 10$. One could prove different bounds were one interested in `large' values of $kT$. In this instance the term $C_{2}$ is not so important, whence one could choose a smaller value of $p$.


\section{The Dedekind zeta-function}\label{zetasec}
This section employs the notation of \S\S \ref{Ds}-\ref{BT}. Consider a number field $K$ with degree $n_{K} = [K: \mathbb{Q}]$ and absolute discriminant $d_{K}$. In addition let $r_{1}$ and $r_{2}$ be the number of real and complex embeddings in $K$, whence $n_{K} = r_{1} + 2r_{2}$. Define the Dedekind zeta-function to be
\begin{equation*}
\zeta_{K}(s) = \sum_{\mathfrak{a} \subset \mathcal{O}_{K}} \frac{1}{(\mathbb{N}\mathfrak{a})^{s}},
\end{equation*}
where $\mathfrak{a}$ runs over the non-zero ideals. The completed zeta-function
\begin{equation}\label{feded}
\xi_{K}(s) = s(s-1) \left(\frac{d_{K}}{\pi^{n_{K}}2^{2r_{2}}}\right)^{s/2} \Gamma(s/2)^{r_{1}} \Gamma(s)^{r_{2}} \zeta_{K}(s)
\end{equation}
satisfies the functional equation
\begin{equation}\label{feeded}
\xi_{K}(s) = \xi_{K}(1-s).
\end{equation}
Let $a(s) = (s-1)\zeta_{K}(s)$ and let
\begin{equation}\label{fdefded}
f(\sigma) = \frac{1}{2}\left\{a(s+ iT)^{N} + a(s-iT)^{N}\right\}.
\end{equation}
It follows from (\ref{feded}) and (\ref{feeded}) that
\begin{equation}\label{DedEE}
\bigg|\Delta_{+} \arg a(s) + \Delta_{-} \arg a(s)\bigg| \leq F(\delta, t)+ n_{K} G(0, \delta, t),
\end{equation}
where $F(\delta, t) = 2\tan^{-1}\frac{1}{2t} - \tan^{-1} \frac{1/2 + \delta}{t} - \tan^{-1}\frac{1/2 - \delta}{t}$, and $G(0, \delta, t)$ is defined in (\ref{Gdef}).

Thus, following the arguments in \S\S \ref{Ds}-\ref{LemSec}, one arrives at
\begin{equation}\label{NDed}
\bigg|N_{K}(T) - \frac{T}{\pi}\log\left\{ d_{K} \left( \frac{T}{2\pi e}\right)^{n_{K}}\right\}\bigg| \leq \frac{2(n+1)}{N} + \frac{2n_{K}}{\pi}\left\{ |g(0, T)| + \log \zeta(\sigma_{1})\right\} + 2,
\end{equation}
where $n$ is bounded above by (\ref{jenco1}), in which $f(s)$ is defined in (\ref{fdefded}). Using the right inequality in
\begin{equation}\label{dedin}
\frac{\zeta_{K}(2\sigma)}{\zeta_{K}(\sigma)}\leq |\zeta_{K}(s)|\leq \left\{\zeta(\sigma)\right\}^{n_{K}},
\end{equation}
one can show that the corresponding estimate for $J_{1}$ is
\begin{equation}\label{JDedFin}
J_{1}/N \leq \pi\log T + \int_{-\pi/2}^{\pi/2} \left\{\log \tilde{w}(T, \phi, \eta, r) + n_{K} \log \zeta( 1 + \eta + r(\tfrac{1}{2} + \eta)\cos \phi)\right\}\, d\phi
\end{equation}
where
\begin{equation}\label{wdert}
\tilde{w}(T, \phi, \eta,r)^{2} =1+ \frac{2r(\frac{1}{2} + \eta)\sin\phi}{T} + \frac{r^{2}(\frac{1}{2} + \eta)^{2} + \eta^{2} + 2r\eta(\frac{1}{2} + \eta)\cos\phi}{T^{2}}.
\end{equation}
For $\phi\in[0, \pi/2]$, the function $\tilde{w}(T, \phi, \eta, r)$ is decreasing in $T$; for $\phi\in[-\pi/2, 0]$ it is bounded above by $\tilde{w}^{*}(T, \phi, \eta,r)$ where
\begin{equation}\label{wder2t}
\tilde{w}^{*}(T, \phi, \eta,r)^{2} =1+ \frac{r^{2}(\frac{1}{2} + \eta)^{2} + \eta^{2} + 2r\eta(\frac{1}{2} + \eta)\cos\phi}{T^{2}}.
\end{equation}
which is decreasing in $T$.

The integral $J_{2}$ is estimated using the following convexity result.
\begin{Lem}\label{convLem}
Let $-\eta\leq p<0$. For $p \leq 1+ \eta - r(\frac{1}{2} + \eta)$ the following bound holds
\begin{equation*}
\begin{split}
|a(s)|^{1 + \eta - p} \leq \left(\frac{1 -p}{1 +p}\right)^{1+ \eta - \sigma} &\zeta_{K}(1+ \eta)^{\sigma - p} \zeta_{K}(1-p)^{1 + \eta - \sigma} |1+ s|^{1 + \eta - p}\\
&\times  \left\{d_{K} \left(\frac{|1 + s|}{2\pi}\right)^{n_{K}}\right\}^{(1 + \eta - \sigma)(1/2 - p)}.
\end{split}
\end{equation*}
\end{Lem}
\begin{proof}
See \cite[\S 7]{Rademacher}. When $p = -\eta$ the bound reduces to that in \cite[Thm 4]{Rademacher}.
\end{proof}
Using this it is straightforward to show that
\begin{equation}\label{FinJ2}
\begin{split}
J_{2}/N &\leq \frac{2r(\frac{1}{2} + \eta)}{1 + \eta - p} \left\{ \log \frac{\zeta_{K}(1 -p)}{\zeta_{K}(1 + \eta)} + \log \frac{1 - p}{1 +p} + (1/2 - p)\log \frac{d_{K}}{(2\pi)^{n_{K}}}\right\} \\
& + \pi \log\zeta_{K}(1 + \eta) + \log T\left(\pi + \frac{2rn_{K}(\frac{1}{2} + \eta)(\frac{1}{2} - p)}{1 + \eta - p}\right)\\
& + \int_{\pi/2}^{3\pi/2} \log w(T_{0}, \phi, \eta, r)\left(1 + \frac{n_{K}r(\frac{1}{2} + \eta)(\frac{1}{2} - p)(-\cos\phi)}{1 + \eta - p}\right)\, d\phi\\
\end{split}
\end{equation}
The quotient of Dedekind zeta-functions can be dispatched easily enough using
\begin{equation*}
-\frac{\zeta_{K}'}{\zeta_{K}} (\sigma) \leq n_{K}\left\{ -\frac{\zeta'}{\zeta}(\sigma)\right\}
\end{equation*}
to show that
\begin{equation*}
\log\frac{\zeta_{K}(1-p)}{\zeta_{K}(1 + \eta)}  = \int_{1-p}^{1 + \eta} -\frac{\zeta_{K}'}{\zeta_{K}}(\sigma)\, d\sigma \leq n_{K}\int_{1-p}^{1 + \eta} -\frac{\zeta'}{\zeta}(\sigma)\, d\sigma \leq n_{K} \log\frac{\zeta(1-p)}{\zeta(1 + \eta)}.
\end{equation*}
Finally the term $-\log|f(1 + \eta)|$ is estimated as in the Dirichlet $L$-function case --- cf.\ (\ref{Rost}). This shows that 
\begin{equation*}\label{RosDed}
\log|f(1+ \eta)| \geq N\log\frac{\zeta_{K}(2 + 2\eta)}{\zeta_{K}(1 + \eta)} + \frac{N}{2}\log(\eta^{2} + T^{2}) + o(1).
\end{equation*}
This, together with (\ref{NDed}), (\ref{JDedFin}), (\ref{wdert}), (\ref{wder2t}) and (\ref{FinJ2}) and sending $N\rightarrow\infty,$ shows that, when $T\geq T_{0}$,
\begin{equation}\label{FinalD}
\begin{split}
\bigg|N_{K}(T) - \frac{T}{\pi}\log\left\{ d_{K} \left( \frac{T}{2\pi e}\right)^{n_{K}}\right\}\bigg| &\leq \frac{r(\frac{1}{2} + \eta)(\frac{1}{2} - p)}{\pi\log r(1 + \eta - p)}\left\{ \log d_{K} + n_{K}\log T\right\}\\
& + \left(C_{2} -\frac{2}{\pi}\left[g(1, T) - |g(0, T)|\right]\right)n_{K} + D_{3},
\end{split}
\end{equation}
where $C_{2}$ is given in (\ref{C2f1}) and
\begin{equation}\label{FinalD2}
\begin{split}
D_{3} &= 2 + \frac{r(\frac{1}{2} + \eta)}{\pi\log r (1 + \eta - p)} \log \left(\frac{1 -p}{1 +p}\right) + \frac{1}{\pi}F( \sqrt{2}(\tfrac{1}{2} + \eta), T_{0}) \\
&+ \frac{1}{2\pi \log r}\bigg( \int_{-\pi/2}^{0} \log \tilde{w}^{*}(T_{0}, \phi, \eta, r) \, d\phi + \int_{0}^{\pi/2} \log \tilde{w}(T_{0}, \phi, \eta, r) \, d\phi\\
&\qquad\qquad\qquad+ \int_{\pi/2}^{\pi} \log w(T_{0}, \phi, \eta, r) \, d\phi + \int_{\pi}^{3\pi/2} \log w^{*}(T_{0}, \phi, \eta, r) \, d\phi\bigg)
\end{split}
\end{equation}
If one chooses $p = -\eta/7$, to ensure that the lower order terms in (\ref{FinalD}) are smaller than those in \cite{NgKad2012}, one arrives at Theorem \ref{ZetaThm}. One may choose a smaller value of $p$ if one is less concerned about the term $D_{2}$.



\section*{Acknowledgements}
I should like to thank Professor Heath-Brown and Professors Ng and Kadiri for their advice. I should also like to thank Professor Giuseppe Molteni and the referee for some constructive remarks.

\bibliographystyle{plain}
\bibliography{themastercanada}

\end{document}